\documentclass[12pt]{article}%
   
\usepackage{amsmath,enumerate}
\usepackage{amsfonts}
\usepackage{amssymb}

\newtheorem{theorem}{Theorem}[section]

\newtheorem{conjecture}[theorem]{Conjecture}

\newtheorem{definition}[theorem]{Definition}

\newtheorem{lemma}[theorem]{Lemma}

\newenvironment{proof}[1][Proof]{\noindent\textbf{#1.} }
{\hfill \ \rule{0.5em}{0.5em}}

\newcommand{\points}{\mathcal{P}}
\newcommand{\lines}{\mathcal{L}}
\newcommand{\incidence}{\mathcal{I}}

\begin{document}

\title{\vspace{-1cm}Independent sets in polarity graphs}

\author{Michael Tait\thanks{Department of Mathematics, University of California San Diego, \texttt{mtait@math.ucsd.edu}} \and Craig Timmons\thanks{Department of Mathematics and Statistics, California State University Sacramento, \texttt{craig.timmons@csus.edu}. This work was supported by a grant from the Simons Foundation (\#359419, Craig Timmons). }}
\date{}
\maketitle
\vspace{-5mm}
\begin{abstract}
Given a projective plane $\Sigma$ and a polarity $\theta$ of $\Sigma$, the corresponding polarity graph is the graph whose vertices are the points of $\Sigma$, and two distinct points $p_1$ and $p_2$ are adjacent if $p_1$ is incident to $p_2^{ \theta}$ in $\Sigma$.  A well-known example of a polarity graph is the Erd\H{o}s-R\'{e}nyi orthogonal polarity graph $ER_q$, which appears frequently in a variety of extremal problems.  Eigenvalue methods provide an upper bound on the independence number of any polarity graph.  Mubayi and Williford showed that in the 
case of $ER_q$, the eigenvalue method gives the correct upper bound in order of magnitude.  We prove that this is also true for other families of polarity graphs.  This includes a family of polarity graphs for which the polarity is neither orthogonal nor unitary.  We conjecture that any polarity graph of a projective plane of order $q$ has an independent set of size $\Omega (q^{3/2})$.  Some related results are also obtained.  
 
\end{abstract}


\section{Introduction}

The use of finite geometry to construct graphs with interesting properties has a rich history in graph theory.  
One of the most well-known constructions is due to  Brown \cite{b}, and Erd\H{o}s, R\'{e}nyi, and S\'{o}s \cite{ers} who used an orthogonal polarity of a Desarguesian projective plane to give examples of graphs that give an asymptotically tight lower bound on the Tur\'{a}n number of the 4-cycle.  Later these same graphs 
were used to solve other extremal problems in a variety of areas such as Ramsey theory \cite{ar}, \cite{kpr}, hypergraph Tur\'{a}n theory \cite{lv}, and even the Cops and Robbers game on graphs \cite{bb}.  While our focus is not on the graphs of \cite{b} and \cite{ers}, we take a 
moment to define them.  Let $q$ be a power of a prime and $PG(2,q)$ be the projective geometry 
over the 3-dimensional vector space $\mathbb{F}_q^3$. We represent the points of $PG(2,q)$ as 
non-zero vectors $(x_0 , x_1 , x_2)$ where $x_i \in \mathbb{F}_q$.  Two vectors $(x_0 , x_1 , x_2)$ and 
$(y_0 , y_1 , y_2)$ are equivalent if $\lambda (x_0 , x_1 , x_2 ) = (y_0 , y_1 , y_2)$ for some 
$\lambda \in \mathbb{F}_q \backslash \{ 0 \}$.  The 
\emph{Erd\H{o}s-R\'{e}nyi orthogonal polarity graph}, denoted $ER_q$, is the graph whose vertices 
are the points of $PG(2,q)$.  Two distinct vertices $(x_0 , x_1 , x_2)$ and $(y_0 , y_1 , y_2)$ are adjacent 
if and only if $x_0 y_0  + x_1 y_1  +x_2 y_2 = 0$.    

The graph $ER_q$ has been studied as an interesting graph in its own right.  Parsons \cite{parsons}
determined the automorphism group of $ER_q$ and obtained several other results.  In particular, 
Parsons showed that for $q \equiv 1 ( \textup{mod}~4)$, $ER_q$ contains a $\frac{1}{2}(q+1)$-regular 
graph on $\binom{q}{2}$ vertices with girth 5.  This construction gives one of the best known lower bounds on 
the maximum number of edges in an $n$-vertex graph with girth 5.  It is still an open problem to determine 
this maximum, and for more on this problem, see \cite{aksv}; especially their Conjecture 1.7 and the discussion preceding it.  Bachrat\'{y} and \v{S}ir\'{a}\v{n} \cite{bs}
reproved several of the results of \cite{parsons} and we recommend \cite{bs} for a good introduction to the graph 
$ER_q$.  They also used $ER_q$ to construct vertex-transitive graphs with diameter two.  

Along with the automorphism group of a graph, two other important graph parameters are the independence number 
and the chromatic number.  
At this time we transition to a more general setting as many of the upper bounds 
on the independence number of $ER_q$ are true for a larger family of graphs.  Let $\Sigma = ( \points , \lines , \incidence)$ be a projective plane of order $q$.  A bijection $\theta : \points \cup \lines \rightarrow \points \cup \lines$ is 
a \emph{polarity} if $\theta ( \points ) = \lines $, $\theta ( \lines ) = \points$, $\theta^2$ is the identity map, and 
$p \incidence l$ if and only if $l^{ \theta} \incidence p^{ \theta}$.  A point $p \in \points$ is called 
an \emph{absolute point} if $p \incidence p^{ \theta}$.  A classical result of Baer is that any polarity of a 
projective plane of order $q$ has at least $q+1$ absolute points.  A polarity with 
$q+1$ points is called an \emph{orthogonal polarity}, and such polarities exist in the 
Desarguesian projective plane as well as in other non-Desarguesian planes.  For more on polarities see 
\cite{hp}, Chapter 12.  
Given a projective plane $\Sigma = ( \points , \lines , \incidence)$ of order $q$ and an orthogonal polarity $\theta$, the corresponding \emph{orthogonal polarity graph} $G( \Sigma , \theta)$ is the graph with vertex set $\points$ where 
two distinct vertices $p_1$ and $p_2$ are adjacent if and only if $p_1 \incidence p_2^{ \theta}$.  
Let $G^{ \circ} ( \Sigma , \theta)$ be the graph obtained from $G( \Sigma , \theta)$ by adding loops to the 
absolute points of $\theta$.  
The integer $q+1$ is an eigenvalue of $G^{ \circ} ( \Sigma , \theta )$ with multiplicity 1, and all other eigenvalues are $\sqrt{q}$ or $- \sqrt{q}$.  A result of Hoffman \cite{hoffman} implies 
\begin{equation}\label{evalue bound}
\alpha ( G^{\circ}  ( \Sigma , \theta) ) \leq \frac{ - (q^2 + q + 1) (-q^{1/2}) }{ q + 1 - \sqrt{q}  }
\end{equation}
which gives $\alpha ( G ( \Sigma  , \theta) ) \leq q^{3/2} + q^{1/2}  + 1$.  
An improved estimate in the case that $q$ is even was obtained by Hobart and Williford \cite{hw} 
using association schemes.  They conjectured that the upper bound (\ref{evalue bound}) can be improved to 
\[
\alpha ( G ( \Sigma , \theta )) \leq q ( q^{1/2} + 1) - 2 ( q^{1/2} - 1 ) ( q + q^{1/2} )^{1/2}
\]
but this is still open.  
Mubayi and Williford \cite{mw} showed that the upper bound 
(\ref{evalue bound}) gives the correct order of magnitude for $\alpha (ER_q)$.  One of the results of \cite{mw} is that 
\begin{equation}\label{eq:mw}
\alpha (ER_{q^2}) \geq \frac{1}{2} q^{3} + \frac{1}{2} q^2 + 1 
\end{equation}
whenever $q$ is a power of an odd prime.  
Their construction can be adapted in a straightforward manner to 
obtain the following lower bound on the independence number of wider class of orthogonal polarity graphs which we introduce now and will be the focus of much of our investigations.  We remark that the study of polarity graphs coming from 
non-Desarguesian planes was suggested in \cite{bs}.  

Let $q$ be a power of an odd prime and $f(X) \in \mathbb{F}_q [X]$.  The polynomial $f(X)$ is a  \emph{planar 
polynomial} if for each $a \in \mathbb{F}_q^*$, the map 
\[
x \mapsto f(x + a) - f(x) 
\]
is a bijection on $\mathbb{F}_q$.  Planar polynomials were introduced by Dembowski and Ostrom \cite{do} in their study of 
projective planes of order $q$ that admit a collineation group of order $q^2$.  Given a planar polynomial 
$f (X) \in \mathbb{F}_q [X]$, one can construct a projective plane as follows.  
Let $\mathcal{P} = \{ (x,y) : x,y \in \mathbb{F}_q \} \cup \{ (x) : x \in \mathbb{F}_q \} \cup \{ ( \infty ) \}$.  
For $a,b,c \in \mathbb{F}_q$, let 
\begin{equation*}
\begin{split}
[ a , b ] =  \{ (x, f(x-a)+b)  :x \in \mathbb{F}_q \} \cup \{ (a) \}, \\
[c] =  \{ ( c , y ) : y \in \mathbb{F}_{q} \} \cup \{ ( \infty ) \},  \\
[ \infty ] =  \{ ( c ) : c \in \mathbb{F}_{q} \} \cup \{ ( \infty ) \}.
\end{split}
\end{equation*} 
Let $\mathcal{L} = \{ [a,b] : a ,b \in \mathbb{F}_q \} \cup \{ [c] :c \in \mathbb{F}_q \} \cup \{ [ \infty ] \}$.  
Define $\Pi_f$ to be the incidence structure whose points are $\mathcal{P}$, whose lines are 
$\mathcal{L}$, and incidence $\mathcal{I}$ is given by containment.  When $f$ is a planar polynomial, 
$\Pi_f$ is a projective plane.  For instance if $f(X) = X^2$ and $q$ is any power of an odd prime, 
$\Pi_f$ is isomorphic to the Desarguesian plane $PG(2,q)$.  For other examples, see \cite{cm}.    

Assume that $f(X) \in \mathbb{F}_q [X]$ is a planar polynomial.  
The plane $\Pi_f$ possesses an orthogonal polarity $\omega$ given by 
\begin{center}
$( \infty )^{ \omega }= [ \infty ], ~~~~ [ \infty ]^{ \omega } = ( \infty ), ~~~~ (c)^{ \omega } = [-c], 
~~~~ [c]^{ \omega} = ( -c )$

\medskip

$(x,y)^{ \omega } = [-x , -y] , ~~\mbox{and} ~~ [a,b]^{ \omega } = (-a,-b)$
\end{center}
where $a,b,c \in \mathbb{F}_q$.  We write $G_f$ for the corresponding orthogonal polarity graph.  This is the graph 
whose vertices are the points of $\Pi_f$ and two distinct vertices $p_1$ and $p_2$ are adjacent in 
$G_f$ if and only if $p_1$ is incident to $p_2^{ \omega}$ in $\Pi_f$.  
For vertices of the form $(x , y)$ the adjacency relation is easily described in terms of $f$.  The distinct 
vertices $(x_1 , y_1)$ and $(x_2 , y_2)$ are adjacent if and only if 
\[
f(x_1 + x_2 ) = y_1 + y_2.
\]
      
Our first result is a generalization of (\ref{eq:mw}) to orthogonal polarity graphs which need not come from a 
Desarguesian plane.  

\begin{theorem}\label{th0}
If $q$ is a power of an odd prime and $f (X) \in \mathbb{F}_{q^2} [X]$ is a planar polynomial all of whose coefficients belong 
to the subfield $\mathbb{F}_q$, then 
\[
\alpha ( G_f ) \geq \frac{1}{2}q^2 ( q - 1)
\]
\end{theorem}

Even though we have the restriction that the coefficients of $f$ belong to $\mathbb{F}_q$, many of the known 
examples of planar functions have this property.
Most of the planar functions discussed in \cite{cm}, including those that give rise to the famous Coulter-Matthews plane, satisfy 
our requirement.  
  
It is still an open problem to determine an asymptotic formula for the independence number of 
$ER_p$ for odd prime $p$.  However, given the results of \cite{mw} and Theorem \ref{th0}, it would be quite surprising 
to find an orthogonal polarity graph of a projective plane of order $q$ whose independence number 
is $o(q^{3/2})$.  We believe that the lower bound $\Omega (q^{3/2})$ is a property shared by 
all polarity graphs, including polarity graphs that come from polarities which are not orthogonal.  

\begin{conjecture}\label{ind conj}
If $G( \Sigma , \theta )$ is a polarity graph of a projective plane of order $q$, then 
\[
\alpha (G ( \Sigma , \theta) ) = \Omega (q^{3/2} ).
\]
\end{conjecture}

There are polarity graphs which are not orthogonal polarity graphs for which Conjecture \ref{ind conj} holds.  
If $G( \Sigma , \theta )$ is a polarity graph where $\theta$ is unitary and $\Sigma$ has order $q$, then 
$\alpha ( G ( \Sigma , \theta ) ) \geq q^{3/2} + 1$.  Indeed, the absolute points of any polarity graph form an independent set 
and a unitary polarity has $q^{3/2} + 1$ absolute points.  In Section \ref{division ring plane} we show that 
there is a polarity graph $G( \Sigma , \theta)$ where $\theta$ is neither orthogonal or unitary and Conjecture 
\ref{ind conj} holds.  

\begin{theorem}\label{non unitary orthogonal}
Let $p$ be an odd prime, $n \geq 1$ be an integer, and $q = p^{2n}$.  There is a polarity graph $G( \Sigma , \theta)$ such that 
$\Sigma$ has order $q$, $\theta$ is neither orthogonal nor unitary, and 
\[
\alpha ( G( \Sigma , \theta ) ) \geq \frac{1}{2}q ( \sqrt{q} - 1 ).
\]
\end{theorem}

In connection with Theorem \ref{th0} and Conjecture \ref{ind conj}, we would like to mention the work
of De Winter, Schillewaert, and Verstra\"{e}te \cite{dwsv} and Stinson \cite{stinson}.  In these papers the 
problem of finding large sets of points and lines such that there is no incidence between these sets is investigated.  
Finding an independent set in a polarity graph is related to this problem as an edge in a polarity graph corresponds to an incidence in the geometry.  The difference is that when one finds an independent set in a polarity graph, choosing the points determines the lines.  In \cite{dwsv} and \cite{stinson}, one can choose the points and lines independently.  

As mentioned above, Conjecture \ref{ind conj} holds for unitary polarity graphs as the absolute points form an 
independent set.  Mubayi and Williford \cite{mw} asked whether or not there is an independent set in 
the graph $U_q$ of size $\Omega (q^{3/2} )$ that contains no absolute points.  For $q$ a square of a prime power, 
the graph $U_q$ has the same vertex set as $ER_q$ and two distinct vertices $(x_0 , x_1, x_2)$ and 
$(y_0 , y_1 , y_2)$ are adjacent if and only if $x_0 y_0^{ \sqrt{q} } +  x_1 y_1^{ \sqrt{q} } +x_2 y_2^{ \sqrt{q} } =0$.
We could not answer their question, but we were able to produce an independent set of size $\Omega ( q^{5/4} )$ that 
contains no absolute points.  We remark that a lower bound of $\Omega (q)$ is trivial.

\begin{theorem}\label{unitary}
Let $q$ be an even power of an odd prime.  The graph $U_q$ has an independent set 
$I$ that contains no absolute points and 
\[
|I |  \geq 0.19239 q^{5/4} - O(q).
\]
\end{theorem}

Related to the independence number is the chromatic number.   In \cite{ptt}, it is shown that 
$\chi ( ER_{q^2} ) \leq 2q + O ( \frac{q}{ \log q} )$ whenever $q$ is a power of an odd prime.  
Here we prove that this upper bound holds for another family of orthogonal polarity graphs.

\begin{definition}
Let $p$ be an odd prime.  Let $n$ and $s$ be positive integers such that $s < 2n$ and $\frac{2n}{s}$ is an odd integer.  
Let $d = p^s$ and $q= p^n$.  We call the pair $\{q,d \}$ an 
\emph{admissible pair}.  
\end{definition}

If $\{ q , d \}$ is an admissible pair, then the polynomial 
$f(X) = X^{d+1} \in \mathbb{F}_{q^2}[X]$ is a planar polynomial.  
For a nice proof, see Theorem 3.3 of \cite{cm}. 

\begin{theorem}\label{th1}
Let $q$ be a power of an odd prime and $\{q , d \}$ be an admissible pair.  If $f(X) = X^{d + 1}$, then 
\[
\chi ( G_f ) \leq 2q + O \left( \frac{q}{ \log q} \right).
\]
\end{theorem}

The eigenvalue bound (\ref{evalue bound}) gives a lower bound of $\chi (G_f ) \geq \frac{q^4 + q^2 + 1}{q^3 + q + 1}$ so that the leading term in the upper bound of Theorem \ref{th1} is best possible up to a constant factor.  Not only does this bound 
imply that $\alpha (G_f) \geq \frac{1}{2} q^3 - o(q^3)$, but shows that most of the vertices of $G_f$ can be partitioned into 
large independent sets.    

The technique that is used to prove Theorem \ref{th1} is the same as the one used in \cite{ptt} and can 
be applied to other orthogonal polarity graphs.  In Section \ref{coloring 2}, we sketch an argument that 
the bound of Theorem \ref{th1} also holds for a plane coming from a Dickson commutative division ring (see \cite{hp}).  
It is quite possible that the technique applies to more polarity graphs, but in order to obtain a general result, some new ideas 
will be needed.  Furthermore, showing that every polarity graph of a projective plane of order $q$ has chromatic number 
at most $O( \sqrt{q} )$ is a significant strengthening of Conjecture \ref{ind conj}.  When $p$ is prime, 
it is still unknown whether or not $\chi (ER_p) = O( \sqrt{p} )$.  


\section{Proof of Theorem \ref{th0}}

Let $q$ be a power of an odd prime and $f(X) \in \mathbb{F}_{q^2} [X]$ be a planar polynomial, all of whose coefficients 
are in the subfield $\mathbb{F}_q$.  Let 
$G_f$ be the orthogonal polarity graph whose construction is given before the statement of Theorem \ref{th0}.  
Partition $\mathbb{F}_q^*$ into two sets $\mathbb{F}_q^+$ and $\mathbb{F}_q^-$ where 
$a \in \mathbb{F}_q^+$ if and only if $-a \in \mathbb{F}_q^-$.  Let $\mu$ be a root of an irreducible quadratic over 
$\mathbb{F}_q$ and so $\mathbb{F}_{q^2} = \{ a + \mu b : a , b \in \mathbb{F}_q \}$.  Let
\[
I = \{ ( x , y + z \mu ) : x , y \in \mathbb{F}_q , z \in \mathbb{F}_q^+ \}.
\]
Note that $|I| = \frac{1}{2} q^2 ( q - 1)$ and we claim that $I$ is an independent set.  Suppose 
$(x_1 , y_1 + z_1 \mu )$ and $(x_2 , y_2 + z_2 \mu )$ are distinct vertices in $I$ and that they are adjacent. Then
\begin{equation}\label{pp}
f(x_1 + x_2) = y_1 + y_2 + ( z_1 + z_2 ) \mu.
\end{equation}
The left-hand side of (\ref{pp}) belongs to $\mathbb{F}_q$ since the coefficients of $f$ are in $\mathbb{F}_q$ and 
$x_1 + x_2 \in \mathbb{F}_q$.
The right-hand side of (\ref{pp}) is not in $\mathbb{F}_q$ since $z_1 + z_2 \neq 0$.  We have a contradiction 
so no two vertices in $I$ are adjacent.  


\section{Proof of Theorem \ref{non unitary orthogonal}}\label{division ring plane}

Let $p$ be an odd prime, $n \in \mathbb{N}$, and $q = p^{2n}$.  Let $\{1 , \lambda \}$ be a basis for a 2-dimensional 
vector space over $\mathbb{F}_q$.  Let $\sigma : \mathbb{F}_q \rightarrow \mathbb{F}_q$ be the map
$x^{\sigma} = x^{ p^n}$.  Observe that $\sigma $ is a field automorphism of order 2, and the fixed elements of $\sigma$ are precisely the elements of the subfield $\mathbb{F}_{p^n}$ in $\mathbb{F}_q$.  
Let $\theta$ be a generator of $\mathbb{F}_q^*$ which is the group of non-zero elements of $\mathbb{F}_q$ under multiplication.  
Let $D$ be the division ring whose elements are 
$\{ x + \lambda y : x , y \in \mathbb{F}_q \}$ where addition is done componentwise, and multiplication is given by the rule 
\[
(x + \lambda y ) \cdot (z + \lambda t) = xz + \theta t y^{ \sigma} + \lambda ( y z + x^{ \sigma} t ).
\]
Here we are following the presentation of \cite{hp}.  Define the map $\alpha : D \rightarrow D$ by 
\[
(x + \lambda y)^{ \alpha } = x +\lambda y^{\sigma}.
\]
Let $\Pi_D = ( \points , \lines , \incidence )$ be the plane coordinatized by $D$.  
That is,
\[
\points = \{ ( x , y ) :x , y \in D \} \cup \{ ( x) : x \in D \} \cup \{ ( \infty ) \}
\]
and 
\[
\lines = \{ [m,k] : m , k \in D \} \cup \{ [m] : m \in D \} \cup \{ [ \infty ] \}
\]
where 
\begin{equation*}
\begin{split}
[m,k] = \{ (x,y) : m \cdot x + y = k \} \cup \{ ( m ) \}, \\
[k] = \{ ( k , y ) : y \in D \} \cup \{ ( \infty ) \},  \\
[ \infty ] = \{ (m ) : m \in D \} \cup \{ ( \infty ) \}.
\end{split}
\end{equation*} 
The incidence relation $\incidence$ is containment.
A polarity of $\Pi_D$ is given by the map $\omega$ where 
\[
( \infty )^{ \omega } = [ \infty], ~~~ [ \infty ]^{ \omega } = ( \infty ),
~~~  (m)^{ \omega } = [ m^{ \alpha} ],~~~ [k]^{ \omega } = ( k^{\alpha } )
\]
and 
\[
(x,y)^{ \omega } = [ x^{ \alpha} , -y^{ \alpha } ] ,~~~~~ [m,k] = ( m^{ \alpha } , -k^{ \alpha } ).
\]
The polarity $\omega$ has $|D|^{5/4} + 1$ absolute points.  
Let $G ( \Pi_D , \omega)$ be the corresponding polarity graph.  

We now derive an algebraic condition for when the distinct vertices 
\begin{center}
$u=(x_1 +  \lambda x_2 , y_1 +  \lambda y_2 )$ and $v=(z_1  + \lambda z_2 , t_1 +  \lambda t_2 )$ 
\end{center}
are adjacent.  
The vertex $u$ is adjacent to $v$ if and only if $u \incidence v^{ \omega}$.  
This is equivalent to  
\[
(x_1 +  \lambda x_2 , y_1 +  \lambda y_2 ) \incidence 
[ z_1 + \lambda  (z_2)^{ \sigma} , - t_1 + \lambda  ( - (t_2)^{ \sigma } ) ]
\]
which in turn, is equivalent to 
\begin{equation}\label{eq1}
(z_1 +  \lambda (z_2)^{ \sigma } ) \cdot (x_1 + \lambda x_2  ) = - y_1 - t_1 + \lambda  (-y_2 - (t_2)^{ \sigma } ).
\end{equation}
Using the definition of multiplication in $D$, (\ref{eq1}) can be rewritten as
\[
 z_1 x_1 + \theta x_2 z_2 + \lambda (  (z_2)^{ \sigma} x_1 +  (z_1)^{ \sigma } x_2 ) 
= - y_1 - t_1 + \lambda  (-y_2 - (t_2)^{ \sigma } ) .
\]
This gives the pair of equations 
\begin{equation}\label{eq2}
x_1 z_1 + \theta x_2 z_2 = - y_1 - t_1
\end{equation}
and 
\[
x_1 (z_2)^{ \sigma } + x_2 ( z_1)^{ \sigma } = - y_2 - (t_2)^{ \sigma}.
\]

Let $\square_q$ be the set of nonzero squares in $\mathbb{F}_q$.  Note that any element of $\mathbb{F}_{p^n}$ is a square in 
$\mathbb{F}_q$.  
Define 
\[
I = \{ (x_1 +  \lambda x_2 , y_1 + \lambda y_2  ) : x_1 , y _1 \in \mathbb{F}_{p^n}, x_2 \in \square_q, y_2 \in \mathbb{F}_q \}.
\]
Then $|I|  = \frac{1}{2}(q-1)q(p^{n} )^2 = \frac{1}{2}q^2 ( q - 1)$.  We now show that $I$ is an independent set.  
Suppose that $(x_1 + x_2 \lambda , y_1 + y_2 \lambda )$ and $(z_1 + z_2 \lambda , t_1 + t_2 \lambda )$ are distinct vertices 
in $I$ that are adjacent.  Then (\ref{eq2}) holds so 
\begin{equation}\label{eq3}
\theta = ( x_2 z_2)^{-1} ( - y_1 - t_1 - x_1 z_1 ).
\end{equation}
The left hand side of (\ref{eq3}) is not a square in $\mathbb{F}_q$.  Since $x_2$ and $z_2$ belong to $\square_q$, 
we have that $(x_2 z_2)^{-1}$ is a square in $\mathbb{F}_q$.  Since $y_1, t_1 , x_1 , z_1 \in \mathbb{F}_{p^n}$, we have that 
$- y_1 - t_1 - x_1 z_1$ is in $\mathbb{F}_{p^n}$ and thus is a square in $\mathbb{F}_q$.  We conclude that the right hand side 
of (\ref{eq3}) is a square.  This is a contradiction and so $I$ must be an independent set.  This shows that 
\[
\alpha (  G ( \Pi_D , \omega ) ) \geq \frac{1}{2} q^2 ( q - 1).
\]


\section{Proof of Theorem \ref{unitary}}

Let $p$ be an odd prime, $n \in \mathbb{N}$, and $q = p^{2n}$.  Let $\theta$ be a generator of $\mathbb{F}_q^*$.
The field $\mathbb{F}_q$ contains a subfield with $\sqrt{q}$ elements and we write 
$\mathbb{F}_{ \sqrt{q} }$ for this subfield.  We will use the fact that $x^{ \sqrt{q} }  = x$ if and only if  
$x \in \mathbb{F}_{ \sqrt{q} }$ and that the characteristic of $\mathbb{F}_q$ is a divisor of $\sqrt{q}$ without explicitly saying so.  

Let $U_q$ be the graph whose vertex set is $V(ER_q)$ and two vertices $(x_0 , x_1 , x_2)$ and $(y_0 , y_1, y_2)$ are adjacent if and only if 
\[
x_0 y_0^{ \sqrt{q} } + x_1 y_1^{ \sqrt{q} } + x_2 y_2^{ \sqrt{q} } = 0.
\]
In \cite{mw}, it is shown that $U_q$ has an independent set $J$ of size $q^3 + 1$.  This independent set consists of the absolute points in $U_q$; namely
\[
J= \{ (x_0 , x_1 , x_2) : x_0^{ \sqrt{q} + 1} + x_1^{ \sqrt{q} + 1} + x_2^{ \sqrt{q} + 1} = 0 \} .
\]

To find an independent set in $U_q$ with no absolute points and size $\Omega (q^{5/4})$, we will work with a graph that is isomorphic to $U_q$.  
Let $U_q^*$ be the graph whose vertex set is $V( ER_q)$ where $(x_0 , x_1 , x_2)$ and 
$(y_0 , y_1 , y_2)$ are adjacent if and only if 
\[
x_0 y_2^{ \sqrt{q}} + x_2 y_0^{ \sqrt{q}} = x_1 y_1^{ \sqrt{q}}.
\]
The proof of Proposition 3 of \cite{mw} is easily adapted to prove the following.  

\begin{lemma}\label{uq:lemma0}
The graph $U_q$ is isomorphic to the graph $U_q^*$.
\end{lemma}

For any $\mu \in \mathbb{F}_q \backslash \mathbb{F}_{ \sqrt{q} }$, we have 
$\mathbb{F}_q = \{ a + b \mu : a , b \in \mathbb{F}_{ \sqrt{q} } \}$.  The next lemma shows that we can find a $\mu$ that 
makes many of our calculations significantly easier.  

\begin{lemma}\label{uq:lemma1}
There is a $\mu \in \mathbb{F}_q \backslash \mathbb{F}_{ \sqrt{q} }$ such that $\mu^{ \sqrt{q} } + \mu = 0$.
\end{lemma}
\begin{proof}
Let $\mu = \theta^{ \frac{1}{2} ( \sqrt{q} + 1)}$.
Since $\mathbb{F}_{ \sqrt{q} }^* = \langle \theta^{ \sqrt{q} + 1} \rangle$, we have that 
$\mu \notin \mathbb{F}_{ \sqrt{q}}$.  Using the fact that $-1 = \theta^{ \frac{1}{2} (q -1) }$, we find that 
\begin{eqnarray*}
\mu^{ \sqrt{q} } + \mu &=& \theta^{ \frac{1}{2} \sqrt{q} ( \sqrt{q} + 1) } + \theta^{ \frac{1}{2} ( \sqrt{q} + 1) } 
 =  \theta^{ \frac{1}{2} \sqrt{q} ( \sqrt{q} + 1) } - \theta^{ \frac{1}{2} (q - 1)  + \frac{1}{2} ( \sqrt{q} + 1) } \\ 
& = & \theta^{ \frac{1}{2} ( q + \sqrt{q} ) } - \theta^{ \frac{1}{2} (q + \sqrt{q} ) } = 0.
\end{eqnarray*}
\end{proof}

\bigskip

For the rest of this section we fix a $\mu \in \mathbb{F}_q \backslash \mathbb{F}_{ \sqrt{q}}$ that satisfies the statement of Lemma \ref{uq:lemma1}.
Given $c \in \mathbb{F}_{ \sqrt{q} }$, define
\[
X_c = \{ ( 1 , a  , b + c \mu )  : a , b \in \mathbb{F}_{ \sqrt{q} } \}.
\]

\begin{lemma}\label{uq:lemma2}
If $c_1$ and $c_2$ are elements of $\mathbb{F}_{ \sqrt{q} }$ with $c_1 \neq c_2$, then the graph $U_q^*$ has no edge with one 
endpoint in $X_{c_1}$ and the other in $X_{c_2}$.
\end{lemma}
\begin{proof}
Suppose that $(1, a_1  , b_1 + c_1 \mu )$ is adjacent to 
$(1 , a_2  , b_2 + c_2 \mu )$ where $a_i , b_i , c_i \in \mathbb{F}_{ \sqrt{q} }$.  By definition of adjacency in $U_q^*$, 
\[
b_2 + c_2 \mu^{ \sqrt{q} } + b_1 + c_1 \mu = a_1 a_2 .
\]
By Lemma \ref{uq:lemma1}, this can be rewritten as 
\begin{equation}\label{uq:lemma2:eq1}
(c_1 - c_2 ) \mu = a_1 a_2  - b_1 - b_2.
\end{equation}
The right hand side of (\ref{uq:lemma2:eq1}) belongs to the subfield $\mathbb{F}_{ \sqrt{q} }$.  Therefore, $c_1 - c_2 = 0$ 
since $\mu \notin \mathbb{F}_{ \sqrt{q} }$.  
\end{proof}

\bigskip

Now we consider the subgraph $U_q^* [ X_c ]$ where $c \in \mathbb{F}_{ \sqrt{q} }$.  The vertex set of 
$U_q^* [ X_c]$ is 
\[
\{ ( 1 , a  , b + c \mu ) : a , b  \in \mathbb{F}_{ \sqrt{q} } \}
\]
and two vertices $(1 , a_1  , b_1 + c \mu )$ and $(1 , a_2  , b_2  + c \mu )$ are adjacent if and only if 
\[
b_2 + c( \mu^{ \sqrt{q} } + \mu ) + b_1 = a_1 a_2  .
\]
By Lemma \ref{uq:lemma1}, this is equivalent to 
\begin{equation}\label{uq:adj1}
b_2 -  a_1 a_2 + b_1 = 0.
\end{equation}

Let $ER_{ \sqrt{q}}^*$ be the graph whose vertex set is $V( ER_{ \sqrt{q} } )$ and 
$(x_0 , x_1 , x_2)$ is adjacent to $(y_0, y_1 , y_2)$ if and only if 
\[
x_0 y_2 -  x_1 y_1 + x_2 y_0 = 0.
\]
Proposition 3 of \cite{mw} shows that $ER_{ \sqrt{q} }^*$ is isomorphic to $ER_{ \sqrt{q} }$.  It follows from (\ref{uq:adj1}) that the graph $U_q^* [ X_c ]$ is isomorphic to the subgraph of $ER_{ \sqrt{q}}^*$ induced by 
$\{ ( 1 , x_1 , x_2 ) : x_1 , x_2 \in \mathbb{F}_{ \sqrt{q} } \}$.  Note that $ER_{ \sqrt{q}}^*$ has exactly 
$\sqrt{q} + 1$ vertices more than $U_q^* [X_c]$.  By Theorem 5 of \cite{mw}, we can find an independent set 
in $U_q^* [ X_c ]$ with at least $.19239q^{3/4} - q^{1/2} - 1$ vertices.  Call this independent set $I_c$.  

We want to throw away the absolute points in $U_q^*$ that are in $I_c$.  
In $U_q^* [X_c]$, the vertex $(1,a  , b + c \mu )$ is an absolute point if and only if 
\[
b  + c \mu + b + c \mu^{ \sqrt{q} } = a^2 
\]
which, again by Lemma \ref{uq:lemma1}, is equivalent to 
\[
2b =  a^2.
\]
There are $\sqrt{q}$ choices for $a$ and a given $a$ uniquely determines $b$.  Thus $I_c$ contains at most 
$q^{1/2}$ absolute points in $U_q$.  Let $J_c$ be the set $I_c$ with the absolute points removed so that 
$|J_c| \geq .19239 q^{3/4}  -2 q^{1/2}-1$.
 
Define 
\[
I = \bigcup_{ c \in \mathbb{F}_{ \sqrt{q} } } J_c.
\]
By Lemma \ref{uq:lemma2}, $I$ is an independent set in $U_q^*$.  Observe that 
\[
|I| \geq  \sqrt{q} ( 0.19239q^{3/4} - 2q^{1/2} - 1) = 0.19239q^{5/4} - O ( q )
\]
and $I$ contains no absolute points.   

We note that when $q$ is a fourth power, the coefficient $0.19239$ may be raised to $\frac{1}{2}$, as Theorem 5 in \cite{mw} is stronger in this case.


\section{Proof of Theorem \ref{th1}}\label{coloring 1}

Let $s$ and $n$ be positive integers with $\frac{2n}{s} = r \geq 3$ an odd integer.  Let $q = p^n$, $d = p^s$, and note 
that $\{q , d \}$ is an admissible pair.  
Let $\mathbb{F}_{q^2}^*$ be the non-zero elements of $\mathbb{F}_{q^2}$ and $\theta$ be a generator of the cyclic group $\mathbb{F}_{q^2}^*$.  Write $\mathbb{F}_q$ and $\mathbb{F}_d$ for the unique 
subfields of $\mathbb{F}_{q^2}$ of order $q$ and $d$, respectively.   
An identity that will be used is 
\[
\frac{ p^{2n} - 1}{p^s - 1} = (p^{s})^{r-1} + (p^s )^{r-2} + \dots + p^s + 1.
\]
It will be convenient to let 
\begin{equation}\label{def of t}
t = \frac{ p^{2n} - 1}{p^s - 1}
\end{equation}
and observe that $t$ is odd since $r  = \frac{2n}{s}$ is odd. 

\begin{lemma}\label{lemma1}
There is a $\mu \in \mathbb{F}_{q^2} \backslash \mathbb{F}_q$ such that when $\mu^d$ is written in the form 
$\mu^d = u_1 + u_2 \mu$ with $u_1 , u_2 \in \mathbb{F}_q$, the element $u_2$ is a $(d-1)$-th power.
\end{lemma}
\begin{proof}
Let $h(X) = X^d + ( \theta^{q + 1} )^{d - 1} X$.  We claim that the roots of $h$ are the elements in the set
$Z=\{ 0 \} \cup \{  \theta^{ q + 1 + i t} : 0 \leq i \leq d - 2 \}$.  
Clearly 0 is a root.  Let $0 \leq i \leq d - 2$.  Note that since $2n = sr$,
\begin{eqnarray*}
dt & \equiv & p^s \left( (p^s)^{r-1} + (p^s)^{r-2} + \dots + p^s + 1\right) \\
&  \equiv &p^{2n} + (p^s)^{r-1} + \dots + p^{2s}  + p^s  \\
& \equiv & 1 + (p^s)^{r-1} + \dots + p^{2s}  + p^s \equiv t (\textup{mod}~p^{2n} - 1).
\end{eqnarray*}
This implies $d( q+ 1 + it) \equiv (q + 1)(d-1) + (q + 1)  + it ( \textup{mod}~q^2 - 1)$
so that 
\[
( \theta^{ q + 1 + it })^d - ( \theta^{q+1})^{d-1} \theta^{q+1 + it} = 0.
\]
We conclude that the roots of $h$ are the elements in $Z$.  

Let $\mu = \theta^{q+1 + t}$.  The non-zero elements of the subfield $\mathbb{F}_{q}$ are the elements 
of the subgroup $\langle \theta^{q+1} \rangle$ in $\mathbb{F}_{q^2}^*$.  Since $t$ is odd and $q+1$ is even, 
$q+1 + t$ is not divisible by $q+1$ thus $\mu \notin \mathbb{F}_q$.  Let $u_2 = ( \theta^{ q + 1} )^{d-1}$ and $u_1 = 0$.
We have 
\[
0  = h( \mu) = \mu^d - ( \theta^{q+1})^{d-1} \mu = \mu^d - u_1 - u_2 \mu
\]
so $\mu^d = u_1 + u_2 \mu$.  By construction, $\mu \in \mathbb{F}_{q^2} \backslash \mathbb{F}_q$, 
$\mu^d = u_1 + u_2 \mu$ with $u_1, u_2 \in \mathbb{F}_{q}$, and $u_2$ is a $(d-1)$-th power.  
\end{proof}
  
\bigskip

The next lemma is known (see Exercise 7.4 in \cite{ln}).  A proof is included for completeness.  

\begin{lemma}\label{lemma2}
If $u_2 , \delta \in \mathbb{F}_{q^2}^*$ and $u_2$ is a $(d-1)$-th power, then for 
any $\xi \in \mathbb{F}_{q^2}$, the equation 
\[
X^d + u_2  \delta^{d-1} X = \xi
\]
has a unique solution in $\mathbb{F}_{q^2}$.  
\end{lemma}
\begin{proof}
Let $u_2, \delta \in \mathbb{F}_{q^2}^*$ and $g(X) = X^d + u_2 \delta^{d-1} X$.
The polynomial $g$ is a permutation polynomial if and only if the only root of $g$ is 0 (see Theorem 7.9 of \cite{ln}).  
If $g(X) = 0$, then $X ( X^{d-1} + u_2 \delta^{d-1} ) = 0$.
It suffices to show that $-u_2 \delta^{d-1}$ is not a $(d-1)$-th power of any element of $\mathbb{F}_{q^2}$ as this would imply that the equation $X^{d-1} + u_2 \delta^{d-1}=0$ has no solutions.  By hypothesis, $u_2 = w^{d-1}$ for some 
$w \in \mathbb{F}_{q^2}^*$.  Since $-1$ is not a $(d-1)$-th power, the product 
$- u_2 \delta^{d-1} = - (w \delta )^{d-1}$ is not a $(d-1)$-th power.  We conclude that 
$g$ is a permutation polynomial on $\mathbb{F}_{q^2}$.  In particular, given any $\xi \in \mathbb{F}_{q^2}$, there is a unique solution to the equation $X^d + u_2 \delta^{d-1} X = \xi$.    
\end{proof}

\bigskip

For the rest of this section, we fix a $\mu \in \mathbb{F}_{q^2} \backslash \mathbb{F}_q$ that satisfies the conclusion of 
Lemma \ref{lemma1}; that is, 
\[
\mu^d = u_1 + u_2 \mu
\]
where $u_1, u_2 \in \mathbb{F}_q$ and $u_2$ is a $(d-1)$-th power in $\mathbb{F}_{q^2}$.  Let 
\[
\mu^{d+1} = w_1 + w_2 \mu
\]
where $w_1, w_2 \in \mathbb{F}_q$.    
We fix a partition of $\mathbb{F}_q^*$ into two sets 
\begin{equation}\label{partition}
\mathbb{F}_q^* = \mathbb{F}_q^+ \cup \mathbb{F}_q^-
\end{equation}
where $a \in \mathbb{F}_q^+$ if and only if $-a \in \mathbb{F}_q^-$.  

It will be convenient to work with a graph that is isomorphic to a large induced subgraph of $G_f$.  By Lemma 
\ref{lemma2}, the map $x \mapsto x^d + x$ is a permutation on $\mathbb{F}_{q^2}$.  Therefore, 
every element of $\mathbb{F}_{q^2}$ can be written in the form $a^d + a$ for some $a \in \mathbb{F}_{q^2}$ and this 
representation is unique.  
Let $\mathcal{A}_{q^2 , d}$ be the graph with vertex set $\mathbb{F}_{q^2} \times \mathbb{F}_{q^2}$ 
where distinct vertices $(a^d + a , x )$ and $(b^d + b , y)$ are adjacent if and only if 
\[
a^d b + a b^d = x + y.
\]
Working with this equation defining our adjacencies will be particularly helpful for the rather technical Lemma \ref{lemma7}
below.  

\begin{lemma}\label{iso}
The graph $\mathcal{A}_{q^2,d}$ is isomorphic to the subgraph of $G_f$ induced by $\mathbb{F}_{q^2} \times \mathbb{F}_{q^2}$.
\end{lemma}
\begin{proof}
One easily verifies that the map $\tau : V( \mathcal{A}_{q^2 , d} ) \rightarrow \mathbb{F}_{q^2} \times \mathbb{F}_{q^2}$ defined by 
\[
\tau ( ( a^d + a , x) ) = (a , x + a^{d+1} )
\]
is a graph isomorphism from $\mathcal{A}_{q^2 ,d}$ to the subgraph of $G_f$ induced by $\mathbb{F}_{q^2} \times 
\mathbb{F}_{q^2}$.  
\end{proof}

\begin{lemma}\label{lemma3}
If 
\[
I^+ = \{ ( a^d + a , x_1 + x_2 \mu ) : a , x_1 \in \mathbb{F}_q , x_2 \in \mathbb{F}_q^+ \},
\]
then $I^+$ is an independent set in the graph $\mathcal{A}_{q^2,d}$.  
The same statement holds with $I^-$ and $\mathbb{F}_q^-$ in place of $I^+$ and $\mathbb{F}_q^+$, respectively.
\end{lemma}
\begin{proof}
Suppose that $(a^d + a , x_1 + x_2 \mu )$ is adjacent to $( b^d + b , y_1 + y_2 \mu)$
where $a, b \in \mathbb{F}_q$.  The left hand side of 
\[
a^d b + a b^d = (x_1 + y_1) + (x_2 + y_2) \mu
\]
is in $\mathbb{F}_q$ so $x_2 + y_2 = 0$.  If $x_2 , y_2 \in \mathbb{F}_q^+$,
 then $x_2 + y_2 \neq 0$ and so no two vertices in $I^+$ can be adjacent.
Similarly, no two vertices in $I^-$ can be adjacent.   
\end{proof}

\begin{lemma}\label{lemma4}
For any $k = \alpha^d + \alpha \in \mathbb{F}_{q^2}$, the map 
\[
\phi_k ( (a^d + a , x) ) = ( a^d + a + k , x + a^d \alpha + a \alpha^d + \alpha^{d+1} )
\]
is an automorphism of the graph $\mathcal{A}_{q^2,d}$.  
\end{lemma}
\begin{proof}
The vertex $\phi_k ( ( a^d + a , x ) )$ is adjacent to 
$\phi_k ( (b^d + b , y ) )$ if and only if 
\begin{equation*}\label{l3:1}
v=( a^d + a + \alpha^d + \alpha , x + a^d \alpha + a \alpha^d + \alpha^{d+1} ) 
\end{equation*}
is adjacent to 
\begin{equation*}\label{l3:2}
w=( b^d + b + \alpha^d + \alpha , y + b^d \alpha + b \alpha^d + \alpha^{d+1} ).
\end{equation*}
Since $d$ is a power of $p$, $v=(  ( a + \alpha)^d + ( a + \alpha ) , x + a^d \alpha + a \alpha^d + \alpha^{d+1} )$.
Similarly, 
\[
w=(  ( b + \alpha)^d + ( b + \alpha ) , y + b^d \alpha + b \alpha^d + \alpha^{d+1} ).
\]
From this we see that $v$ is adjacent to $w$ if and only if 
\begin{equation}\label{eq:v2-1}
\begin{split}
(a + \alpha )^d ( b + \alpha) + (a + \alpha ) (b + \alpha )^d = \\
x + a^d \alpha + a \alpha^d + \alpha^{d+1} + y + b^d \alpha + b \alpha^d + \alpha^{d+1}.
\end{split}
\end{equation}  
A routine calculation shows that (\ref{eq:v2-1}) is equivalent to the equation $a^d b + a b^d = x + y$ which holds if and only if 
$(a^d + a , x)$ is adjacent to $(b^d + b , y )$ in $\mathcal{A}_{q^2 , d}$.    
\end{proof}

\bigskip

Let $J = I^+ \cup I^-$ and observe that 
$J = \{ ( a^d + a , x_1 + x_2 \mu ) : a , x_1 \in \mathbb{F}_q , x_2 \in \mathbb{F}_q^* \}$.
Let 
\[
K = \bigcup_{ \beta \in \mathbb{F}_q } \phi_{ ( \beta \mu )^d + ( \beta \mu ) } ( J).
\]

\begin{lemma}\label{lemma5}
If $\mathcal{A}_{q^2,d} [K]$ is the subgraph of $\mathcal{A}_{q^2,d}$ induced by $K$, then 
\[
\chi ( \mathcal{A}_{q^2,d} [K] ) \leq 2q.
\]
\end{lemma}
\begin{proof}
By Lemma \ref{lemma3}, the vertices in $J$ may be colored using at most 2 colors.  
By Lemma \ref{lemma4}, the vertices in $\phi_{ k } ( J)$ can also be colored using at most 2 colors.  Since $K$ is the union 
of $q$ sets of the form $\phi_k (J)$ where $k \in \mathbb{F}_{q^2}$, we may color $K$ using at most $2q$ colors.  
\end{proof}

\bigskip

Lemma \ref{lemma5} shows that we can color all but at most $O(q^3)$ vertices of $\mathcal{A}_{q^2,d}$ with at most $2q$ colors.  We now show that the remaining vertices can be colored with $o(q)$ colors.  
Before stating the next lemma we recall that $\mu^d = u_1 + u_2 \mu$ and we let $\mu^{d+1} = w_1  + w_2 \mu$ 
where $u_1 , u_2 , w_1 , w_2 \in \mathbb{F}_q$.  

\begin{lemma}\label{lemma6}
If $X = ( \mathbb{F}_{q^2} \times \mathbb{F}_{q^2}) \backslash K$, then 
\[
X =  \{ ( (a + \beta \mu )^d + ( a + \beta \mu) , t_1 + ( a^d \beta + a \beta^d u_2 + \beta^{d+1} w_2 ) \mu ) : a , \beta , t_1 \in \mathbb{F}_q \}.
\]
\end{lemma}
\begin{proof}
For any $\beta \in \mathbb{F}_q$, the set $\phi_{ ( \beta \mu)^d + ( \beta \mu ) } (J)$ can be written as 
\[
\{ (a^d + a + ( \beta \mu )^d + ( \beta \mu) , 
x_1 + x_2 \mu + a^d ( \beta \mu) + a ( \beta \mu)^d + ( \beta \mu)^{d+1} ) : a , x_1 \in \mathbb{F}_q , x_2 \in \mathbb{F}_q^* \}.
\]
Let $(s_1 + s_2 \mu , t_1 + t_2 \mu ) \in \mathbb{F}_{q^2} \times \mathbb{F}_{q^2}$ where $s_1,s_2,t_1,t_2 \in \mathbb{F}_q$.  
The vertex $(s_1 + s_2 \mu , t_1 + t_2 \mu )$ is in $K$ if we can 
find $a , x_1 , \beta \in \mathbb{F}_q$ and $x_2 \in \mathbb{F}_q^*$ such that 
\begin{eqnarray}
s_1 + s_2 \mu &  = & ( a + \beta \mu )^d + a + \beta \mu, \label{eq11} \\ 
t_1 + t_2 \mu & = & x_1 + x_2 \mu + a^d ( \beta \mu )  + a ( \beta \mu )^d + ( \beta \mu )^{d+1}. \label{eq12}
\end{eqnarray}
Since every element of $\mathbb{F}_{q^2}$ can be written as $z^d + z$ for some $z \in \mathbb{F}_{q^2}$, we can 
write $s_1 + s_2 \mu = z^d + z$ and then choose $a$ and $\beta$ in $\mathbb{F}_q$ so that 
$z = a + \beta \mu$.  With this choice of $a$ and $\beta$, equation (\ref{eq11}) holds.  

Since $\mu^{d+1} = w_1 + w_2 \mu$, equation (\ref{eq12}) can be rewritten as
\begin{equation}\label{eq13}
t_1 + t_2 \mu = 
(x_1 + a \beta^d u_1 + \beta^{d+1} w_1 ) + 
(x_2 + a^d \beta + a \beta^d u_2 + \beta^{d+1} w_2 ) \mu.
\end{equation}
Let $x_1 = t_1 - a \beta^d u_1 - \beta^{d+1} w_1$.  If 
$t_2 \neq a^d \beta + a \beta^d u_2 + \beta^{d+1} w_2$, then we can take 
$x_2 = t_2 - a^d \beta - a \beta^d u_2 - \beta^{d+1} w_2$ and (\ref{eq12}) holds.  
Therefore, the vertices in $\mathbb{F}_{q^2} \times \mathbb{F}_{q^2}$ not in $K$ are those vertices in the set
\[
\{ ( (a + \beta \mu )^d + ( a + \beta \mu) , t_1 + ( a^d \beta + a \beta^d u_2 + \beta^{d+1} w_2 ) \mu ) : a , \beta , t_1 \in \mathbb{F}_q \}.
\]
\end{proof}

\begin{lemma}\label{lemma7}
If $\mathcal{A}_{q^2,d} [X]$ is the subgraph of $\mathcal{A}_{q^2,d}$ induced by $X= ( \mathbb{F}_{q^2} \times 
\mathbb{F}_{q^2} ) \backslash K$, then 
\[
\chi ( \mathcal{A}_{q^2,d} [X] ) = O \left( \frac{q}{ \log q} \right).
\]
\end{lemma}
\begin{proof}
For $\beta \in \mathbb{F}_q$, partition $X$ into the sets $X_{\beta}$ where 
\[
X_{ \beta } = \{ ( (a + \beta \mu )^d + ( a + \beta \mu) , t_1 + ( a^d \beta + a \beta^d u_2 + \beta^{d+1} w_2 ) \mu ) : a , t_1 \in \mathbb{F}_q \}.
\]
Fix a $\beta \in \mathbb{F}_q$ and a vertex  
\[
v=( (a + \beta \mu)^d + ( a + \beta \mu ) , t_1 + ( a^d \beta + a \beta^d u_2 + \beta^{d+1} w_2 ) \mu )
\]
in $X_{ \beta}$.  Let $\gamma \in \mathbb{F}_q$.  We want to count the number 
of vertices 
\[
w=( (x + \gamma \mu)^d + ( x + \gamma \mu ) , y_1 + ( x^d \gamma + x \gamma^d u_2 + \gamma^{d+1} w_2 ) \mu )
\]
in $X_{ \gamma}$ that are adjacent to $v$.  The vertices $v$ and $w$ are adjacent if and only if 
\begin{equation}\label{adj eq}
\begin{split}
(a + \beta \mu)^d ( x + \gamma \mu) + (a + \beta \mu ) (x + \gamma \mu )^d \\
= t_1 + y_1 + ( a^d \beta + a \beta^d u_2 + \beta^{d+1} w_2 + x^d \gamma + x \gamma^d u_2 + \gamma^{d+1} w_2 ) \mu.
\end{split}
\end{equation}
If $\gamma = \beta$, then we can choose $x \in \mathbb{F}_q$ in $q$ different ways and the above equation uniquely determines 
$y_1$.  We conclude that the vertex $v \in X_{ \beta }$ has at most $q$ other neighbors in $X_{ \beta}$.  

Assume now that $\gamma \neq \beta$.  We need to count how many $x , y_1 \in \mathbb{F}_q$ satisfy 
(\ref{adj eq}).  
A computation using the relations
$\mu^d = u_1 + u_2 \mu$ and $\mu^{d+1} = w_1 + w_2 \mu$ shows that (\ref{adj eq}) is equivalent to 
\begin{equation*}
\begin{split}
a^d x + a^d \gamma \mu + \beta^d x ( u_1  + u_2 \mu ) + \beta^d \gamma (w_1 + w_2 \mu) \\
+ ax^d + a \gamma^d ( u_1  + u_2 \mu) + \beta x^d \mu + \beta \gamma^d ( w_1 + w_2 \mu) \\
= t_1 + y_1 + ( a^d \beta + a \beta^d u_2 + \beta^{d+1} w_2 + x^d \gamma + x \gamma^d u_2 + \gamma^{d+1} w_2 ) \mu.
\end{split}
\end{equation*}
Equating the coefficients of $\mu$ gives
\begin{equation*}
a^d \gamma  + \beta^d x u_2 + \beta^d \gamma w_2 + a \gamma^d u_2 + \beta x^d + \beta \gamma^d w_2 \\
= a^d \beta + a \beta^d u_2 + \beta^{d+1} w_2 + x^d \gamma + x \gamma^d u_2 + \gamma^{d+1} w_2.
\end{equation*}
This equation can be rewritten as 
\begin{equation}\label{eq21}
x^d ( \gamma - \beta ) + x ( \gamma^d - \beta^d)u_2 = \xi
\end{equation}
for some $\xi \in \mathbb{F}_q$ that depends only on $a$, $\gamma$, $\beta$, and $\mu$.  
Since $\gamma - \beta \neq 0$, equation (\ref{eq21}) is equivalent to 
\begin{equation}\label{eq22}
x^d + u_2 ( \gamma - \beta)^{d-1} x= \xi ( \gamma - \beta )^{-1}.
\end{equation}
By Lemma \ref{lemma2}, (\ref{eq22}) has a unique solution for $x$
since $u_2$ is a $(d-1)$-power and $\gamma - \beta \in \mathbb{F}_q^*$.  Once $x$ is determined, (\ref{adj eq}) gives 
a unique solution for $y_1$.  Therefore, $v$ has at most one neighbor in $X_{ \beta}$.
We conclude that the degree of $v$ in $X$ is at most $q + (q-1) <2q$.  

The graph $\mathcal{A}_{q^2,d}[X]$ does not contain a 4-cycle and has maximum degree at most $2q$.  This implies that the 
neighborhood of any vertex contains at most $q$ edges.  
By a result of Alon, Krivelevich, and Sudakov \cite{aks}, the graph $\mathcal{A}_{q^2,d} [X]$ can be colored using  
$O \left ( \frac{q}{ \log q} \right)$ colors.  
\end{proof}

\bigskip

\begin{proof}[Proof of Theorem \ref{th1}]
Partition the vertex set of $\mathcal{A}_{q^2,d}$ as 
\[
V( \mathcal{A}_{q^2,d} ) = K \cup X.
\]
By Lemmas \ref{lemma5} and \ref{lemma7}, we can color the vertices in $K \cup X$ using 
$2q + O \left( \frac{q}{ \log q} \right)$ colors.  This gives a coloring of the vertices 
in $\mathbb{F}_{q^2} \times \mathbb{F}_{q^2}$ in $G_f$ and it only remains to color the vertices in the set 
$\{ (m) : m \in \mathbb{F}_{q^2} \} \cup \{ ( \infty )  \}$.    

The vertex $(\infty)$ is adjacent to $(m)$ for every $m \in \mathbb{F}_{q^2}$.  
Since $G_f$ is $C_4$-free, 
the subgraph of $G_f$ induced by the neighborhood of $( \infty )$ induces a 
a graph with maximum degree at most 1.  
We may color 
the vertices in $\{ ( m) : m \in \mathbb{F}_{q^2} \} \cup \{ ( \infty ) \}$ using three new colors not used to color 
$ \mathbb{F}_{q^2} \times \mathbb{F}_{q^2}$ to obtain a 
$2q + O ( \frac{q}{ \log q} )$ coloring of $G_f$.  
\end{proof}


\section{Dickson Commutative Division Rings}\label{coloring 2}

Let $p$ be an odd prime, $n > 1$ be an integer, $q = p^n$, and $a$ be any element of $\mathbb{F}_{q}$ that is not a square.  
Let $1 \leq r < n$ be an integer.  Let $D$ be a 2-dimensional vector space over $\mathbb{F}_q$ with 
basis $\{1, \lambda \}$.  Define a product $\cdot$ on $D$ by the rule 
\[
(x + \lambda y ) \cdot ( z + \lambda t) = xz + a y^{p^r} t^{p^r} + \lambda ( yz + st).
\]
With this product and the usual addition, $D$ is a commutative division ring (see \cite{hp}, Theorem 9.12 and note that it is common to call such a structure a semifield).  
We can use $D$ to define a projective plane $\Pi$ (see \cite{hp}, Theorem 5.2).      
This plane also has an orthogonal polarity (see \cite{hp}, page 248).  
Let $\mathcal{GD}_{q^2}$ be the corresponding orthogonal polarity graph.  Using the argument of Section 
\ref{coloring 1}, one can prove that 
\[
\chi ( \mathcal{GD}_{q^2} ) \leq 2q + \left( \frac{ q}{ \log q} \right).
\]
A rough outline is as follows.  Let 
$\mathcal{AD}_{q^2}$ be the subgraph of $\mathcal{GD}_{q^2}$ induced by 
the vertices 
\[
\{ ( (x_1 + \lambda x_2 , y_1 + \lambda y_2 ) : x_1 , x_2 , y_1 , y_2 \in \mathbb{F}_q \}.
\]
Partition $\mathbb{F}_{q}^*$ into the sets $\mathbb{F}_q^+$ and $\mathbb{F}_q^-$ where 
$a \in \mathbb{F}_q^+$ if and only if $-a \in \mathbb{F}_q^-$.  
The sets 
\[
I^+ = \{ ( x_2 \lambda , y_1 + y_2 \lambda ) : x_2, y_1 \in \mathbb{F}_q , y_2 \in \mathbb{F}_q^+ \}
\]
and 
\[
I^- = \{ ( x_2 \lambda , y_1 + y_2 \lambda ) : x_2, y_1 \in \mathbb{F}_q , y_2 \in \mathbb{F}_q^- \}
\]
are independent sets in $\mathcal{GD}_{q^2}$.
For any $k \in \mathbb{F}_q$, the map 
\[
\phi_k ( x_1 + \lambda x_2 , y_1 + \lambda y_2 ) = 
(x_1 + \lambda x_2 + k , y_1 + \lambda y_2 + k x_1 + 2^{-1} k^2 + \lambda x_2 k )
\]
is an automorphism of $\mathcal{AG}_{q^2}$.

Let $J = I^+ \cup I^-$ and $K = \bigcup_{k \in \mathbb{F}_q } \phi_k (J)$ and observe that  
\[
K = \{ ( k + x_2 \lambda , y_1 + y_2 \lambda + 2^{-1} k^2 + \lambda x_2 k ) : x_1 , y_1 , k \in \mathbb{F}_q , y_2 \in \mathbb{F}_{q}^{*} \}.
\]
If $X = ( D \times D ) \backslash K$, then 
\[
X = \{ ( s_1 + s_2 \lambda , t_1 + (s_2 s_1) \lambda ) : s_1 , s_2 , t_1 \in \mathbb{F}_q \}.
\]
It can then be shown that the subgraph of $\mathcal{GD}_{q^2}$ induced by $X$ has maximum degree at most $2q$.  
The remaining details are left to the reader.  


\section{Concluding Remark}

The argument used to prove Theorem \ref{unitary} can be extended to other unitary polarity graphs.  We illustrate with an example.
Let $a$ and $e$ be integers with $a \not\equiv \pm ( \textup{mod}~2e )$, $e \equiv 0 ( \textup{mod}~4)$, and $\textup{gcd}(a,e) = 1$.  Let $f: \mathbb{F}_q \rightarrow \mathbb{F}_q$ be the polynomial $f(X) = X^n$ where $n = \frac{1}{2} ( 3^a + 1)$ and 
$q = 3^e$.  The map $f$ is a planar polynomial and the corresponding plane is the Coulter-Matthews plane \cite{cm}.
This plane has a unitary polarity whose action on the affine points and lines is given by 
\[
(x ,y)^{ \theta} = [ - x^{ \sqrt{q} } , - y^{ \sqrt{q} } ] ~~ \mbox{and} ~~ (a,b)^{ \theta} = ( - a^{ \sqrt{q} } , - b^{ \sqrt{q}} ).
\]
The proof of Theorem \ref{unitary} can be modified to show that the corresponding unitary polarity graph has an independent set 
of size $\frac{1}{2} q^{5/4} - o( q^{5/4} )$ that contains no absolute points.  The reason for the condition 
$e \equiv 0 ( \textup{mod}~4)$ instead of $e \equiv 0 ( \textup{mod}~2)$, which is the condition given in \cite{cm} for $f$ to be planar, is that we need $\sqrt{q}$ to be a square in order to apply Theorem \ref{th0} to the subgraphs that correspond 
to the $U_q^* [X_c]$ in the proof of Theorem \ref{unitary}.



\begin{thebibliography}{10}
 

\bibitem{aksv}
P.\ Allen, P.\ Keevash, B.\ Sudakov, J.\ Verstra\"{e}te,
Tur\'{a}n numbers of bipartite graphs plus an odd cycle,
{\em J.\ Combin.\ Theory, Ser.\ B} 106 (2014), 134-162.  

\bibitem{aks}
N.\ Alon, M.\ Krivelevich, B.\ Sudakov,
Coloring graphs with sparse neighborhoods,
{\em J.\ Combin.\ Theory, Ser.\ B} 77 (1999), 73-82.

\bibitem{ar}
N.\ Alon, V.\ R\"{o}dl, 
Sharp bounds for some multicolor Ramsey numbers,
{\em Combinatorica} 25 (2005) no.\ 2, 125-141.  

\bibitem{bs}
M.\ Bachrat\'{y}, J.\ \v{S}ir\'{a}\v{n},
Polarity graphs revisited,
{\em Ars Math.\ Contemp}.\ 8 (2015), no.\ 1, 55-67.  

\bibitem{bb}
A.\ Bonato, A.\ Burgess, 
Cops and robbers on graphs based on designs,
{\em J.\ Combin. Des}.\ 21 (2013), no.\ 9, 404-418.  

\bibitem{b}
W.\ G.\ Brown,
On graphs that do not contain a Thomsen graph,
{\em Canad.\ Math.\ Bull}.\ {\bf 9} (1966) 281-289.

\bibitem{cm}
R.\ Coulter, R.\ Matthews, 
Planar functions and planes of Lenz-Barlotti Class II,
{\em Des.\ Codes Cryptogr}.\ 10 (1997), no.\ 2, 167-184.  


\bibitem{dwsv}
S.\ De Winter, J.\ Schillewaert, J.\ Verstra\"{e}te, 
Large incidence-free sets in geometries, 
{\em Electron.\ J.\ Combin}.\ 19 (2012), no.\ 4, Paper 24.  

\bibitem{demb}
P.\ Dembowski,
{\em Finite Geometries},
Springer-Verlag Berlin Heidelberg, Germany, 1968.  

\bibitem{do}
P.\ Dembowski, T.\ G.\ Ostrom,
Planes of order $n$ with collineation groups of order $n^2$, 
{\em Math.\ Z}.\ {\bf 103} 1968, 239-258.

\bibitem{ers}
P.\ Erd\H{o}s, A.\ R\'{e}nyi, V.\ T.\ S\'{o}s,
On a problem of graph theory,
{\em Studia Sci.\ Math.\ Hungar}.\ {\bf 1} 1966, 215-235.



\bibitem{hw}
S.\ Hobart, J.\ Williford,
The independence number for polarity graphs of even order planes,
{\em J.\ Algebraic\ Combin}.\ 38 (2013), no.\ 1, 57-64.  

\bibitem{hoffman}
A.\ J.\ Hoffman,
On eigenvalues and colorings of graphs, 
1970 {\em Graph Theory and its Applications 
(Proc.\ Advanced Sem., Math.\ Research Center, Univ.\ of Wisconsin, Madison, Wis., 1969)}, 
Academic Press, New York.  

\bibitem{hp}
D.\ R. Hughes, F.\ C.\ Piper,
{\em Projective Planes},
GTM Vol.\ 6, Springer-Verlag New-York-Berlin, 1973.  

\bibitem{kpr}
A.\ Kostochka, P.\ Pudl\'{a}k, V.\ R\"{o}dl,
Some constructive bounds on Ramsey numbers,
{\em J.\ Combin.\ Theory, Ser.\ B} 100 (2010), no.\ 5, 439-445.

\bibitem{lv}
F.\ Lazebnik, J. Verstra\"{e}te,
On hypergraphs of girth five,
{\em Electron.\ J.\ Combin}.\ 10 (2003), \#R25.

\bibitem{ln}
R.\ Lidl, H.\ Niederreiter,
{\em Finite Fields},
Cambridge University Press 1997, 2nd Ed.

\bibitem{mw}
D.\ Mubayi, J.\ Williford,
On the independence number of the Erd\H{o}s-R\'{e}nyi and projective norm graphs and a related hypergraph,
{\em J.\ Graph Theory} 56 (2007), no.\ 2, 113-127.

\bibitem{parsons}
T.\ D.\ Parsons, 
Graphs from projective planes, 
{\em Aequationes Math}.\ 14 (1976), no.\ 1-2, 167-189.  

\bibitem{ptt}
X.\ Peng, M.\ Tait, C.\ Timmons,
On the chromatic number of the Erd\H{o}s-R\'{e}nyi orthogonal polarity graph,
{\em Electron.\ J.\ Combin}.\ 22 (2015) no.\ 2 \#P2.21

\bibitem{stinson}
D.\ Stinson, 
Nonincident points and blocks in designs,
{\em Discrete Math}.\ 313 (2013), no.\ 4, 447-452.

%


\end{thebibliography}
\end{document}